\documentclass[12pt,leqno]{amsart}
\usepackage{amsmath,amssymb,amsfonts}
\usepackage{eucal, enumerate}

\newtheorem{lem}{Lemma}
\newtheorem{cor}[lem]{Corollary}
\newtheorem{thm}[lem]{Theorem}
\theoremstyle{definition}
\newtheorem{defn}{Definition}

\renewcommand{\phi}{\varphi}                 
\renewcommand{\epsilon}{\varepsilon}

\newcommand\G{\Gamma}

\begin{document}

\begin{center}
{\large\sc Characterization of Line-Consistent Signed Graphs}\\[10pt]
by\\[10pt]
Daniel C.\ Slilaty \\
Wright State University\\
Dayton, OH 45435-0001, U.S.A.\\
{\tt daniel.slilaty@wright.edu}
\\[10pt]
and \\[10pt]
Thomas Zaslavsky\\
Binghamton University (SUNY)\\
Binghamton, NY 13902-6000, U.S.A.\\
{\tt zaslav@math.binghamton.edu} 
\\[10pt]
\today
\end{center}
\bigskip\bigskip\bigskip

{\small 
\emph{Abstract.}  
The line graph of a graph with signed edges carries vertex signs.  A vertex-signed graph is \emph{consistent} if every circle (cycle, circuit) has positive vertex-sign product.  Acharya, Acharya, and Sinha recently characterized \emph{line-consistent} signed graphs, i.e., edge-signed graphs whose line graphs, with the naturally induced vertex signature, are consistent.  Their proof applies Hoede's relatively difficult characterization of consistent vertex-signed graphs.  We give a simple proof that does not depend on Hoede's theorem as well as a structural description of line-consistent signed graphs.
\bigskip

\bigskip

}

\bigskip\bigskip\bigskip


In the first article on signed graphs---graphs whose edges are labelled positive or negative---Harary \cite{NB} gave a simple characterization of those in which the product of signs around every circle (i.e., circuit, cycle) is positive.  (Such graphs are called \emph{balanced}.)  Later, Beineke and Harary \cite{BH} introduced signed vertices and asked the analogous question of characterizing the vertex-signed graphs (also called \emph{marked graphs}) in which the product of vertex signs around every circle is positive.  (These vertex-signed graphs are called \emph{consistent}.)  This was more difficult.  After preliminary results by Acharya \cite{A,A2} and Rao \cite{R}, Hoede found a definitive answer \cite{H}, which was developed more deeply in \cite{RX}.  
(An even more definitive answer was found subsequently; see \cite{JSD}.)

An obvious question was never answered until recently.  If a signed graph $\Sigma$ has underlying graph $\G$ and edge signature $\sigma$, then the line graph $L(\G)$ has $\sigma$ as a vertex signature.  Under what conditions is this vertex signature consistent?  We call such a graph \emph{line consistent}, as the defining property is consistency of the line graph.  Acharya, Acharya, and Sinha \cite{AAS} found a simple necessary and sufficient condition for line consistency.  The necessity of their condition is easy to determine.  Sufficiency is not so easy; it depends on Hoede's relatively complex consistency criterion.  Here we give an elementary, short proof of sufficiency as well as a direct structural description of the signed graphs whose line graphs are consistent.  
(The related paper \cite{CNV} gives constructions for line-consistent signed graphs that illuminate more about their structure.)


A graph may have multiple edges but not loops.  A \emph{simple} graph has neither loops nor multiple edges.  
The \emph{degree} $d(v)$ of a vertex is the number of edges incident with $v$ (and also the number of neighbors of $v$ if the graph is simple).  
The length of a path $P$ is denoted by $l(P)$; if it is zero the path is \emph{trivial}.  
\emph{Suppressing} a divalent vertex means replacing it and its two incident edges by a single edge.  
Two graphs are \emph{homeomorphic} if they are isomorphic or they can both be changed into the same (unlabelled) graph by suppressing divalent vertices in one or both of them.
When $e,e'$ are parallel edges in $\G$, the line graph $L(\G)$ has a double edge between its vertices $e$ and $e'$.  

In a signed graph $\Sigma = (\G,\sigma)$, the sign of a circle $C$, written $\sigma(C)$, is the product of its edge signs.  Similarly, the sign $\sigma(W)$ of a path or of any walk $W=e_1e_2\cdots e_l$ is $\sigma(W):=\sigma(e_1)\sigma(e_2)\cdots\sigma(e_l)$.  
A vertex is \emph{totally positive} (or, \emph{totally negative}) if all incident edges are positive (or, negative).  
The \emph{negative subgraph} of $\Sigma$ is the spanning subgraph $\Sigma^-$ whose edges are the negative edges of $\Sigma$.  

The \emph{line graph of a signed graph} $\Sigma$, written $L_\sigma(\Sigma)$, is defined as the vertex-signed graph $(L(\G),\sigma)$ whose underlying graph is $L(\G)$, the line graph of $\G$, and whose vertices are marked by the sign function $\sigma$ of $\Sigma$.  (Other notions of line graph of a signed graph exist, in which edges are signed instead of vertices, but they are not related to this work.)

\begin{defn}\label{D:line-consistency}
A signed graph $\Sigma$ is called \emph{line consistent} if $L_\sigma(\Sigma)$ is consistent.  (This is the same as ``$S$-consistency'' in \cite{AAS}, where $\Sigma$ is called $S$.)
\end{defn}

The main result of \cite{AAS}, Theorem 2.1 (when corrected by revising the first line of part (2) to read ``for every vertex $v_i$ \ldots in $S$ such that $d(v_i) \geq 3$,'' as the authors obviously intended), applies to simple graphs with edge signs.  It states:

\begin{thm}[{\cite[Theorem 2.1]{AAS}}]\label{T:AAS}
A signed simple graph $\Sigma$ is line consistent if and only if it is balanced, every vertex of degree $d(v)>3$ is totally positive, and each vertex of degree $d(v)=3$ is either totally positive or has two negative edges which belong to all circles through the vertex.
\end{thm}


We reformulate Theorem \ref{T:AAS} in a simpler way that leads to a short proof, we add a second and third criterion for line consistency, and we generalize by allowing the underlying graph to have multiple edges.  An \emph{isthmus} (called by some a ``bridge'') is an edge whose deletion raises the number of connected components; equivalently, it is an edge that does not belong to any circle.

\begin{thm}\label{T:characterization}
Each of the following conditions on a signed graph $\Sigma$ without loops is necessary and sufficient for it to be line consistent: 
\begin{enumerate}[{\rm(i)}]
\item $\Sigma$ is balanced, each vertex of degree $d(v)>3$ is totally positive, and each vertex of degree $d(v)=3$ is totally positive or has exactly one positive edge, which is an isthmus.
\item $\Sigma$ is balanced, its negative subgraph is a vertex-disjoint union of paths and circles, and each endpoint $v$ of a negative edge is incident with at most one positive edge, which is an isthmus if $d_{\Sigma^-}(v)=2$.
\item Each vertex $v$ of degree $d(v)>3$ is totally positive, each vertex of degree $d(v)=3$ is totally positive or has exactly one positive edge, and after deleting all positive isthmi, the signed graph is balanced and the endpoints of every negative edge have degree at most $2$.
\end{enumerate}
\end{thm}

\begin{proof}
The equivalence of (i) and (ii) is obvious, so we give short proofs of the necessity and sufficiency of (i) and the equivalence of (ii) and (iii).

To prove necessity of (i), consider the possibilities.  
If a vertex $v$ has either three negative edges, or two positive edges and one negative edge, these edges form a negative triangle in the line graph.  This implies most of (i).  We have to prove that, if $v$ is incident with negative edges $e$ and $e'$ and a positive edge $f$ (and no other edges), then $f$ is an isthmus.  

If not, there is a circle $C$ on $f$ which must contain $e$ or $e'$; let us say $e$.  In the line graph $C$ generates a circle $L(C)$, which is positive because its vertex sign product equals the edge sign product of $C$.  However, in the line graph there is another circle that interposes $e'$ between $e$ and $f$.   This circle is negative.  Hence, $f$ must be an isthmus.  Thus, (i) is necessary for line consistency.

Now we prove sufficiency.  A \emph{digon} is a circle of length 2.  A \emph{vertex triangle} is a circle of length 3 in $L(\G)$ whose vertices are three edges that are incident with a single vertex in $\G$.

A circle $C$ in $L_\sigma(\Sigma)$ has the form $e_le_1\cdots e_{l-1}e_l$ where $l \geq 2$.  A digon or vertex triangle in $L_\sigma(\Sigma)$ is obviously positive.  Any other triangle comes from a triangle in $\Sigma$, so is also positive.  Thus, we may assume $l\geq 4$ and that any shorter circle is positive.

Note that an isthmus $e=uv$ of $\G$ that is not a pendant edge is a cutpoint of $L(\G)$, separating the other edges incident with $u$ from those incident with $v$.  The only way it can belong to a circle in $L(\G)$ is for the preceding and following edges of the circle to be incident with the same endpoint of $e$.

Suppose consecutive edges $e_{i-1}e_ie_{i+1}$ are incident with $v$.  If $e_i$ is positive, $C$ can be shortened by omitting it without changing the sign of the circle.  If $e_i$ is negative, since $d(v)=3$ one of the edges is a positive isthmus, say $e_{i+1}$.  Because $e_{i+1}$ is an isthmus, $e_{i+2}$ must be incident with $v$ as well, which is impossible because $d(v)=3$ and (since $l\geq4$) $e_{i+2}\neq e_{i-1}$.  Therefore, we may assume no three consecutive edges in $C$ are incident with the same vertex.  That means $C = L(W)$ for a closed walk $W$ in $\Sigma$.  But every closed walk is positive because of balance.  Thus, $C$ is positive.

That (ii) implies (iii) is obvious.  Assume (iii); we deduce (ii).  $\Sigma$ is balanced because isthmi do not affect balance.  Also, $\Sigma^-$ has maximum degree at most 2, so it is a disjoint union of paths and circles.  Let $\Sigma'$ be $\Sigma$ without its positive isthmi.  If $d_\Sigma(v)=3$ and $v$ has one positive edge $f$ and two negative edges $e_1, e_2$, then $d_{\Sigma'}(v)\leq2$ so $f$ must be an isthmus.
\end{proof}

It is interesting to see what the theorems say about relatively well connected graphs.

\begin{cor}\label{C:3conn}
Let $\Sigma$ of order at least $4$ be $3$-connected, or just edge $2$-connected without divalent vertices.  Then $\Sigma$ is line consistent if and only if it is all positive.
\hfill\qedsymbol
\end{cor}

Although (ii) and (iii) in the theorem are simple restatements of (i), their different points of view are suggestive.  Parts (ii) and (iii) suggest constructions for line-consistent signed graphs, for which see \cite{CNV}.  

Part (ii) can be interpreted as a structural description of a line-consistent signed graph $\Sigma$.  (Recall that we forbid loops.)  
A \emph{block} is a graph that is connected and has no cutpoint.  A block is \emph{nontrivial} if it contains a circle; thus, the trivial blocks are the isthmi and isolated vertices.  A \emph{block of $\G$} is a maximal block subgraph of $\G$; equivalently, it is a maximal connected subgraph that is not separated by any cutpoint of $\G$.  
We apply the same terminology to $\Sigma$ as to its underlying graph.  

\begin{thm}\label{C:struct}
A signed graph $\Sigma$ is line consistent if and only if it is balanced and has the following form:
\begin{enumerate}[{\rm\ (1)}]
\item Each component of $\Sigma^-$ is a circle, a nontrivial path, or a single vertex.
\item A circle component of $\Sigma^-$ is a block of $\Sigma$ and each of its vertices is incident with at most one other edge, which must be a positive isthmus.
\item A nontrivial path component $P$ of $\Sigma^-$ either is an induced subgraph of $\Sigma$, or is all but one edge of a circle that is a block of $\Sigma$ whose remaining edge is positive.  The endpoints of $P$ are at most divalent.  Each internal vertex of $P$ is incident with at most one other edge, which must be a positive isthmus.  Furthermore, $P$ either
\begin{enumerate}[{\rm(a)}]
\item  is part of a nontrivial block of $\Sigma$ (then its endpoints are necessarily divalent); or
\item  is composed entirely of isthmi and its endpoints are not incident with any nontrivial block of $\Sigma$ (then the second edge, if any, incident with an endpoint is necessarily a positive isthmus).
\end{enumerate}
\end{enumerate}
\end{thm}

\begin{proof}
The form is stronger than the characterization in (ii), so it implies line consistency.  We verify the converse in stages.

The characterization of an all-negative circle follows directly from (ii).  

Consider a nontrivial path component $P$ of the negative subgraph.  

If $P$ is contained in a nontrivial block $B$, each endpoint must be incident with a second edge, which is positive.  Suppose the second edge at both endpoints is the same edge $e$; then $P \cup \{e\}$ is a circle.  No other edge can be incident with the endpoints of $P$, but an internal vertex can be incident with one more edge, which can only be an isthmus.  If the positive edges at the endpoints are distinct, then $P$ is induced because, again, any third edge at an internal vertex is an isthmus, therefore has its other endpoint off $P$.

If $P$ is not contained in a nontrivial block, it is composed entirely of isthmi.  Suppose it were not; then it would have a vertex $v$ that is in a nontrivial block $B$ and is incident with an edge of $P$ that is not in $B$.  The degree of $v$ in $B$ is at least 2; therefore, $d(v)=3$ and by (ii) the third edge at $v$ is a positive isthmus.  However, that leaves only one edge at $v$ that can be in $B$, an impossibility.  Therefore, if $P$ is not entirely within a nontrivial block, its every edge is an isthmus.  By similar reasoning a second edge at an endpoint of $P$ must be an isthmus.
\end{proof}

Acharya, Acharya, and Sinha \cite{AASC} examined a similar problem, where the line graph has edge as well as vertex signs derived from $\Sigma$.  The question is whether the product of edge signs and vertex signs on each circle is the same---this property is called \emph{harmony}.  In \cite{AASC} the edge signs are those of the Behzad--Chartrand line graph \cite{BC}.  Other definitions of a signed line graph (such as that in \cite[Sect.\ 5.2]{MTS}) could be considered, though if the edge signature is balanced (as in the $\times$-line signed graph of M.\ Acharya \cite{TimesLG}) the answer is the same as in Theorem \ref{T:AAS} since switching (defined in, e.g., \cite{MTS}) the line graph's edge signature does not alter the characterization of harmony.



\begin{thebibliography}{99}

\bibitem{A} B.\ Devadas Acharya, 
A characterization of consistent marked graphs.  
\emph{Nat.\ Acad.\ Sci.\ Letters (India)} 6 (1983), 431--440.  
Zbl 552.05052.  

\bibitem{A2} B.\ Devadas Acharya, 
Some further properties of consistent marked graphs.  
\emph{Indian J.\ Pure Appl.\ Math.}\ 15 (1984), 837--842.  
MR 86a:05101.  Zbl 552.05053.

\bibitem{AASC} Belmannu Devadas Acharya, Mukti Acharya, and Deepa Sinha, 
Cycle-compatible signed line graphs.  
\emph{Indian J.\ Math.}\ 50 (2008), no.\ 2, 407--414.
MR 2517744 (2010h:05142).  Zbl 1170.05032.

\bibitem{AAS} B.\ Devadas Acharya, Mukti Acharya, and Deepa Sinha,  
Characterization of a signed graph whose signed line graph is $S$-consistent.  
\emph{Bull.\ Malaysian Math.\ Sci.\ Soc.} (2) 32 (2009), no.\ 3, 335--341.
MR 2562172 (2010m:05135).  Zbl 1176.05032.

\bibitem{TimesLG} Mukti Acharya, 
$\times$-line signed graphs.  
\emph{J.\ Combin.\ Math.\ Combin.\ Comput.}\ 69 (2009), 103--111.
MR 2517311.  Zbl 1195.05031.

\bibitem{BC} M.\ Behzad and G.\ Chartrand, 
Line-coloring of signed graphs.  
\emph{Elem.\ Math.}\ 24 (1969), 49--52.
MR 39 \#5415.  Zbl 175, 503 (e: 175.50302).

\bibitem{BH} Lowell W.\ Beineke and Frank Harary, 
Consistent graphs with signed points.  
\emph{Riv.\ Mat.\ Sci.\ Econom.\ Social.}\ 1 (1978), 81--88.  
MR 81h:05108.  Zbl 493.05053.

\bibitem{NB} F.\ Harary, 
On the notion of balance of a signed graph. 
\emph{Michigan Math.\ J.}  2 (1953--54), 143--146 and addendum preceding p.\ 1.
MR 16, 733h.  Zbl 56, 421c (e: 056.42103).

\bibitem{H} Cornelis Hoede, 
A characterization of consistent marked graphs.  
\emph{J.\ Graph Theory} 16 (1992), 17--23.  
MR 93b:05141.  Zbl 748.05081.

\bibitem{JSD} Manas Joglekar, Nisarg Shah, and Ajit A.\ Diwan,
Balanced group labeled graphs.  
\emph{Discrete Math.}\ 312 (2012), no.\ 9, 1542--1549.
MR 2899887.  Zbl 1239.05162.

\bibitem{R} S.B.\ Rao, 
Characterizations of harmonious marked graphs and consistent nets.  
\emph{J.\ Combin.\ Inform.\ System Sci.}\ 9 (1984), 97--112.
MR 89h:05048.  Zbl 625.05049. 

\bibitem{RX} Fred S.\ Roberts and Shaoji Xu, 
Characterizations of consistent marked graphs.  
\emph{Discrete Appl.\ Math.}\ 127 (2003), 357--371.
MR 1984094 (2004b:05097).  Zbl 1026.05054.

\bibitem{MTS} Thomas Zaslavsky, 
Matrices in the theory of signed simple graphs.  
In: \emph{Advances in Discrete Mathematics and Applications:\ Mysore, 2008}, B.D.\ Acharya, G.O.H.\ Katona, and J.\ Nesetril (Eds.), pp.\ 207--229.  
Ramanujan Math.\ Soc.\ Lect.\ Notes Ser., No.\ 13.  
Ramanujan Math.\ Soc., Mysore, India, 2010.  
MR 2766941 (2012d:05017).  Zbl 1231.05120.

\bibitem{CNV} Thomas Zaslavsky, 
Consistency in the naturally vertex-signed line graph of a signed graph.  
\emph{Bull.\ Malaysian Math.\ Sci.\ Soc.}, to appear. 

\end{thebibliography}
\end{document}